\documentclass[10pt,a4paper]{article}
\usepackage{eurosym}
\usepackage{amsfonts}
\usepackage{amsmath,amssymb,amsthm,multirow,graphicx,ifpdf,mathtools}
\usepackage{fixltx2e}
\usepackage{enumitem}
\usepackage{xcolor}
\usepackage{graphicx}
\usepackage{epstopdf}

\newtheorem{theorem}{Theorem}[section]
\newtheorem{lemma}[theorem]{Lemma}

\newtheorem{corollary}[theorem]{Corollary}

\theoremstyle{
definition}
\newtheorem{definition}[theorem]{Definition}

\newtheorem{example}[theorem]{Example}

\allowdisplaybreaks
\begin{document}

\title{Global attractivity for some classes of Riemann--Liouville fractional
differential systems}
\author{H.T.~Tuan\thanks{%
Institute of Mathematics, Vietnam Academy of Science and Technology, 18
Hoang Quoc Viet, 10307 Ha Noi, Viet Nam, \texttt{httuan@math.ac.vn}} \and %
Adam Czornik\thanks{%
Silesian University of Technology, Faculty of Automatic Control, Electronics
and Computer Science, Akademicka 16, 44-100 Gliwice, Poland, \texttt{%
adam.czornik@polsl.pl}} \and Juan J.~Nieto\thanks{%
Departamento de Estat\'istica, An\'alise Matem\'atica e Optimizaci\'on,
Facultade de Matem\'aticas, Universidade de Santiago de Compostela, 15782
Santiago de Compostela, Spain, \texttt{juanjose.nieto.roig@usc.es}} \and %
Micha\l\ Niezabitowski\thanks{%
Silesian University of Technology, Faculty of Automatic Control, Electronics
and Computer Science, Akademicka 16, 44-100 Gliwice, Poland, \texttt{%
michal.niezabitowski@polsl.pl}}}
\maketitle

\begin{abstract}
In this paper, we present some results for existence of global solutions and
attractivity for mulidimensional fractional differential equations involving
Riemann-Liouville derivative. First, by using a Bielecki type norm and
Banach fixed point theorem, we prove a Picard-Lindel\"of type theorem on the
existence and uniqueness of solutions. Then, applying the properties of
Mittag-Leffler functions, we describe the attractivity of solutions to some
classes of Riemann--Liouville linear fractional differential systems.

 \medskip
  

  \smallskip
  
\end{abstract}

\section{Introduction}

In recent years, fractional-order differential equations have attracted
increasing interests due to their applications in modeling anomalous
diffusion, time-dependent materials and processes with long range
dependence, allometric scaling laws, and complex networks. For more details,
we refer the reader to the monographs e.g. \cite%
{MillerRoss1993,Podlubny,KilbasEtal2006,Diethelm}. In this paper we consider
Riemann-Liouville differential systems%
\begin{equation}
D_{0_{+}}^{\alpha }x(t)=f(t,x(t)),  \label{O}
\end{equation}%
where $\alpha \in \left( 0,1\right) $ and $D_{0_{+}}^{\alpha }$ is the
Riemann-Liouvile derivative of order $\alpha $, $f:\left[ 0,\infty \right)
\times \mathbb{R}^{d}\rightarrow \mathbb{R}^{d}$ is a given function and $%
x:\left( 0,\infty \right) \rightarrow \mathbb{R}^{d}$ is the solution. The
initial value problem for (\ref{O}) we define as a problem of finding a
solution that fulfills the condition%
\begin{equation*}
\underset{t\rightarrow 0^{+}}{\lim }t^{1-\alpha }x\left( t\right) =x_{0}
\end{equation*}%
for an a priori given $x_{0}\in \mathbb{R}^{d}.$ We investigate two
fundamental problems connected to this equation: existence of a unique
global on $\left( 0,\infty \right) $ solution of the initial value problem
and the attractivity of these solutions understood as the property of
tending to zero of each solution. To prove the existence and uniqueness of
the initial value problem, in the space $C_{1-\alpha }$ $\left( \left[
0,\infty \right) ,\mathbb{R}^{s}\right) $ consisting of all continuous
functions $f:\left( 0,\infty \right) \rightarrow \mathbb{R}^{s}$ such that
there exists the limit 
\begin{equation*}
\underset{t\rightarrow 0^{+}}{\lim }t^{1-\alpha }f(t),
\end{equation*}%
we assume that there exists a bounded (or continuous) nonnegative function $%
L:\left[ 0,\infty \right) \rightarrow \left[ 0,\infty \right) $ such that 
\begin{equation}
\left\Vert f\left( t,x\right) -f\left( t,y\right) \right\Vert \leq L\left(
t\right) \left\Vert x-y\right\Vert  \label{A1}
\end{equation}%
for all $t\in \left[ 0,\infty \right) $ and $x,y\in \mathbb{R}^{d}.$ For the
attractivity problem we assume that the function $f$ has the following form%
\begin{equation*}
f(t,x)=Ax(t)+Q\left( t\right) x(t)+g(t).
\end{equation*}%
It seems that the problem for existence of solution of (\ref{O}) was for the
first time considered in \cite{Pitcher} in case of $d=1$ and zero initial
condition under assumption that $f$ is bounded, continuous and Lipschitzian
in the second variable. This result has been extended to arbitrary initial
condition in \cite{Bassam}. Next paper dealing with the initial value
problem for global solution of Riemann-Liouville differential systems is 
\cite{Delbosco}. In this paper the authors consider the one-dimensional case
($d=1$) and they prove that for a stationary system 
\begin{equation*}
D_{0_{+}}^{\alpha }x(t)=f(x(t))
\end{equation*}%
the condition (\ref{A1}) with a constant function $L$ is a sufficient
condition for the existence and uniqueness of the solution of the initial
value problem. After this paper the problem of existence of a global
solution of Riemann-Liouville differential systems has been also considered
in \cite{Koua,Idczak,Trif}. The results of the first paper are still about
scalar equation and provide sufficient conditions for existence of at least
one solution. The conditions are as follows:

1) $f(t,x(t))\in C_{1-\alpha }(\left[ 0,\infty \right) ,\mathbb{R})$ for any 
$x(t)\in C_{1-\alpha }(\left[ 0,\infty \right) ,\mathbb{R})$;

2) there exist three nonnegative continuous functions $p(t),w(t)$ and $q(t)$
defined on $\left[ 0,\infty \right) $, $p(t)$ and $q(t)$ are bounded, such
that%
\begin{align*}
|f(t,x)|&\leq p(t)w\left( \frac{\left\vert x\right\vert }{1+t^{2}}\right)
+q(t), \\
w(t)&<t
\end{align*}%
for all $t\in \left[ 0,\infty \right) $ and $x\in \mathbb{R}$;

3) 
\begin{equation*}
\underset{t\geq 0}{\sup }\int_{0}^{t}\frac{t^{1-\alpha }\left( t-s\right)
^{\alpha -1}s^{\alpha -1}}{1+t^{2}}p(s)ds<\Gamma \left( \alpha \right).
\end{equation*}

Even these conditions guarantee only existence not uniqueness of a solution
of the initial value problem it is difficult to compare them with our ones.
Also for the one-dimensional case there are the results presented in \cite%
{Trif}, where the author proved existence and uniqueness of the global
solution of the initial value problem if the function $f$ has the following
form 
\begin{equation*}
f(t,x)=p(t)x+q(t),
\end{equation*}%
where $p\in C_{\beta }(\left[ 0,\infty \right) ,\mathbb{R}),$ $q\in
C_{1-\alpha }(\left[ 0,\infty \right) ,\mathbb{R})$, $p$ and $q$ are
nonnegative and $0\leq \beta <\alpha .$ The multidimensional version of
existence and uniqueness of the global solution of the initial value problem
has been discussed in \cite{Idczak}. In this paper the authors, motivated by
control theory applications, study this problem in the space of summable
function and show that if the condition (\ref{A1}) with a constant function $%
L$ is satisfied and the function $f\left( \cdot ,0\right) $ is summable,
then there exists a unique global solution of the initial value problem
which is summable. In the above mentioned papers also the results for local
solution are presented and further results of this kind are published in 
\cite{Diethelm} and \cite{Trujillo}.

The problem of attractivity sometimes also called incorrectly asymptotic
stability of solutions of nonlinear Riemann-Liouville differential systems
has been considered in \cite{Qian,Qin,Chen,Zhou}. The authors of \cite{Qian}
consider the right-hand side of (\ref{O}) in the form%
\begin{equation*}
f(t,x)=Ax(t)+Q(t)x(t)
\end{equation*}%
and claim that the stability of the matrix $A$ and boundedness of the
function $Q$ implies attractivity. As it was noticed in \cite{Cong} this is
not true (it is enough to consider $Q(t)=C-A,$ where $C$ is any $d$ by $d$
matrix). In \cite{Chen} a one-dimensional version of the equation (\ref{O})
is investigated. The presented in this paper conditions are however
difficult to check since they are expressed in the terms of properties of
the function $f$ \ as a functional on the set $\left( 0,\infty \right)
\times C\left( \left( 0,\infty \right) ,\mathbb{R}\right) $, where $C\left(
\left( 0,\infty \right) ,\mathbb{R}\right) $ is the set of all continuous
functions $g:\left( 0,\infty \right) \rightarrow \mathbb{R}$. The main
result of \cite{Qin} (Theorem 1) is obtained under the assumptions that 
\begin{equation*}
\left\Vert t^{\alpha -1}E_{\alpha ,\alpha }\left( At^{\alpha }\right)
\right\Vert \leq Me^{-\gamma t}
\end{equation*}%
for all $t\in \left[ 0,\infty \right) $ and certain positive $M$ and $\gamma
.$ However, this assumption is never satisfied since 
\begin{equation*}
\underset{t\rightarrow 0^{+}}{\lim }\left\Vert t^{\alpha -1}E_{\alpha
,\alpha }\left( At^{\alpha }\right) \right\Vert =\infty ,
\end{equation*}%
whereas 
\begin{equation*}
\underset{t\rightarrow 0^{+}}{\lim }Me^{-\gamma t}=1.
\end{equation*}%
Finally, in \cite{Zhou} the case of Riemann-Liouville differential equation
in a Banach space $\left( X,\left\Vert \cdot \right\Vert \right) $ is
considered. The author shows that the hypotheses:

1) $\left\Vert f\left( t,x\right) \right\Vert \leq Lt^{-\beta }\left\Vert
x\right\Vert ^{\delta }$ $\ $for all $t\in \left( 0,\infty \right) $, $x\in
X $ and certain $L\geq 0,$ $\alpha <\beta <1$ and $\delta \in \mathbb{R}$;

2) there exists a constant $\kappa >0$ such that for any bounded set $%
E\subset X,$%
\begin{equation*}
\sigma \left( f\left( t,E\right) \right) \leq \kappa \sigma \left( E\right) ,
\end{equation*}%
where $\sigma $ is the Hausdorf measure of noncompactness, guarantee
existence of at least one attractive solution for the initial value problem.

\section{Preliminaries}

Consider the equation%
\begin{equation}
D_{0_{+}}^{\alpha }x(t)=Ax(t)+Q\left( t\right) x(t)+g(t),  \label{1}
\end{equation}%
where $\alpha \in (0,1)$, and the Riemann--Liouville fractional derivative 
\begin{equation*}
D_{0+}^{\alpha }x(t):=\frac{d}{dt}I_{0+}^{1-\alpha }x(t)
\end{equation*}%
is defined with the Riemann--Liouville fractional integral 
\begin{equation*}
I_{0+}^{\beta }x(t):=\frac{1}{\Gamma (\beta )}\int_{0}^{t}(t-\tau )^{1-\beta
}x(\tau )d\tau ,
\end{equation*}%
for $\beta >0$ and $I_{0+}^{0}x(t):=x(t)$, see \cite[p.~13 \& p.~27]%
{Diethelm} and \cite[p.~2]{Nieto2}. Denote by $C_{1-\alpha }$ $\left([
0,\infty) ,\mathbb{R}^{s}\right) $ the set of all continuous functions $%
f:\left( 0,\infty \right) \rightarrow \mathbb{R}^{s}$ such that there exists
the limit 
\begin{equation*}
\underset{t\rightarrow 0^{+}}{\lim }\;t^{1-\alpha }f(t)
\end{equation*}%
and $C_{1-\alpha }^{0}$ $\left( \left[ 0,\infty \right) ,\mathbb{R}%
^{s}\right) $ the subset of $C_{1-\alpha }$ $\left( \left([ 0,\infty \right)
,\mathbb{R}^{s}\right) $ consisting of all functions $f\in C_{1-\alpha }$ $%
\left( \left[ 0,\infty \right) ,\mathbb{R}^{s}\right) $ satisfying%
\begin{equation*}
\underset{t\geq 0}{\sup }\;t^{1-\alpha }\left\Vert f(t)\right\Vert <\infty,
\end{equation*}%
where $\|\cdot\|$ is an arbitrary norm on $\mathbb{R}^s$. For $f\in
C_{1-\alpha }^{0}$ $\left( \left[ 0,\infty \right) ,\mathbb{R}^{s}\right) $
denote 
\begin{equation*}
\left\Vert f\right\Vert _{1-\alpha }=\underset{t\geq 0}{\sup }t^{1-\alpha
}\left\Vert f(t)\right\Vert .
\end{equation*}%
It is obvious to see that $\left\Vert \cdot \right\Vert _{1-\alpha }$ is a
norm and the space $\left( C_{1-\alpha }^{0}\left( \left[ 0,\infty \right) ,%
\mathbb{R}^{s}\right) ,\left\Vert \cdot \right\Vert _{1-\alpha }\right) $ is
a Banach space. In through this paper, we define 
\begin{equation*}
\Lambda _{\alpha }^{s}:=\left\{ \lambda \in \mathbb{C}\setminus \{0\}:|\arg {%
(\lambda )}|>\frac{\alpha \pi }{2}\right\} .
\end{equation*}%
For any matrix $A\in \mathbb{R}^{s\times s}$, the set $\sigma (A)$ is the
spectrum of $A$, i.e. 
\begin{equation*}
\sigma (A):=\{\lambda \in \mathbb{C}:\lambda \;\text{is an eigenvalue of the
matrix}\;A\}.
\end{equation*}%
Furthermore, we use $|||A|||$ to denote the norm of $A$ respect to the norm $%
\|\cdot\|$ on $\mathbb{R}^s$.

In our further consideration we will use the following result.

\begin{lemma}
Let $A\in \mathbb{R}^{s\times s}$ and suppose that $\sigma(A)\subset
\Lambda_\alpha^s$. Then the following statements are valid.

\begin{itemize}
\item[(i)] There exists $t_{0}>0$ and a positive constant $M$ which depends
on parameters $t_0$, $\alpha$, $A$ such that 
\begin{equation*}
t^{\alpha -1}||| E_{\alpha ,\alpha }(t^{\alpha }A)||| \leq \frac{M}{%
t^{\alpha +1}},\quad \forall t\geq t_{0}.
\end{equation*}

\item[(ii)] The quantity%
\begin{equation*}
t^{1-\alpha }\int_{0}^{t}(t-\tau )^{\alpha -1}||| E_{\alpha ,\alpha
}((t-\tau )^{\alpha }A)||| \tau ^{\alpha -1}d\tau
\end{equation*}
is bounded on $[0,\infty )$, i.e. 
\begin{equation*}
\sup_{t\geq 0}\;t^{1-\alpha }\int_{0}^{t}(t-\tau )^{\alpha -1}||| E_{\alpha
,\alpha }((t-\tau )^{\alpha }A)||| \tau ^{\alpha -1}d\tau <\infty .
\end{equation*}
\end{itemize}
\end{lemma}

\begin{proof}
\noindent (i) The proof is obtained by using \cite[Lemma 4]{Tuan} and
arguments as in \cite[Lemma 5]{Tuan}.

\noindent (ii) Let $T>2t_{0}$ be an arbitrary constant. First, we consider
the case $t\in \lbrack 0,T]$. Due to the fact that $E_{\alpha ,\alpha
}(t^{\alpha }A)$ is continuous on $[0,T]$, we have 
\begin{align*}
t^{1-\alpha }\int_{0}^{t}(t-\tau )^{\alpha -1}||| E_{\alpha ,\alpha
}((t-\tau )^{\alpha }A)||| \tau ^{\alpha -1}d\tau & \leq C_{1}t^{1-\alpha
}\int_{0}^{t}(t-\tau )^{\alpha -1}\tau ^{\alpha -1}d\tau \\
& =C_{1}t^{1-\alpha }t^{2\alpha -1}B (\alpha ,\alpha ) \\
& =C_{1}t^{\alpha }B (\alpha ,\alpha ) \\
& \leq C_{1}T^{\alpha }B (\alpha ,\alpha ),
\end{align*}%
where $C_{1}:=\sup_{t\in \lbrack 0,T]}\Vert E_{\alpha ,\alpha }(t^{\alpha
}A)\Vert $, and $B (\alpha ,\alpha )$ is the Beta function.

For the case $t>T>2t_{0}$, we have 
\begin{align*}
t^{1-\alpha }\int_{0}^{t}(t-\tau )^{\alpha -1}||| E_{\alpha ,\alpha
}((t-\tau )^{\alpha }A)||| \tau ^{\alpha -1}d\tau & \leq t^{1-\alpha
}\int_{0}^{t-t_{0}}(t-\tau )^{\alpha -1}||| E_{\alpha ,\alpha }((t-\tau
)^{\alpha }A)||| \tau ^{\alpha -1}d\tau \\
& \hspace{1cm}+ t^{1-\alpha }\int_{t-t_{0}}^{t}(t-\tau )^{\alpha -1}|||
E_{\alpha ,\alpha }((t-\tau )^{\alpha }A)||| \tau ^{\alpha -1}d\tau \\
& \leq t^{1-\alpha }(I_{1}(t)+I_{2}(t)),
\end{align*}%
where%
\begin{equation*}
I_{1}(t):=\int_{0}^{t-t_{0}}(t-\tau )^{\alpha -1}||| E_{\alpha ,\alpha
}((t-\tau )^{\alpha }A)||| \tau ^{\alpha -1}d\tau
\end{equation*}%
and%
\begin{equation*}
I_{2}(t):=\int_{t-t_{0}}^{t}(t-\tau )^{\alpha -1}||| E_{\alpha ,\alpha
}((t-\tau )^{\alpha }A)||| \tau ^{\alpha -1}d\tau .
\end{equation*}
From (i), we see that 
\begin{align*}
I_{1}(t)& \leq M\int_{0}^{t-t_{0}}\frac{1}{(t-\tau )^{\alpha +1}}\tau
^{\alpha -1}d\tau \\
& =M\int_{0}^{t/2}\frac{1}{(t-\tau )^{\alpha +1}}\tau ^{\alpha -1}d\tau
+M\int_{t/2}^{t-t_{0}}\frac{1}{(t-\tau )^{\alpha +1}}\tau ^{\alpha -1}d\tau
\\
& \leq M\frac{2^{\alpha +1}}{t^{\alpha +1}}\int_{0}^{t/2}\tau ^{\alpha
-1}d\tau +M\left( \frac{t}{2}\right) ^{\alpha -1}\int_{t_{0}}^{t/2}\frac{1}{%
u^{\alpha +1}}du \\
& \leq M\left( \frac{2}{\alpha t}+\frac{2^{1-\alpha }}{\alpha t_{0}^{\alpha
}t^{1-\alpha }}\right) .
\end{align*}%
Thus, 
\begin{align}
t^{1-\alpha }I_{1}(t)& \leq M\left( \frac{2}{\alpha t^{\alpha }}+\frac{%
2^{1-\alpha }}{\alpha t_{0}^{\alpha }}\right)  \notag \\
& \leq M\left( \frac{2}{\alpha T^{\alpha }}+\frac{%
2^{1-\alpha }}{\alpha t_{0}^{\alpha }}\right) .  \label{est1}
\end{align}%
On the other hand, for $t>T>2t_{0}$, we see that 
\begin{align}
t^{1-\alpha }I_{2}(t)& \leq t^{1-\alpha }(t-t_{0})^{\alpha
-1}\int_{0}^{t_{0}}u^{\alpha -1}||| E_{\alpha ,\alpha }(u^{\alpha }A)||| \;du
\notag \\
& \leq 2^{1-\alpha }t_{0}^{\alpha }E_{\alpha ,\alpha +1}(t_{0}^{\alpha }|||
A|||).  \label{est2}
\end{align}%
From \eqref{est1} and \eqref{est2}, we have 
\begin{equation*}
\sup_{t\geq 0}t^{1-\alpha }\int_{0}^{t}(t-\tau )^{\alpha -1}||| E_{\alpha
,\alpha }((t-\tau )^{\alpha }A)||| \tau ^{\alpha -1}\;d\tau <\infty .
\end{equation*}%
The proof is complete.
\end{proof}

\section{Existence and uniqueness of solutions to Riemann--Liouville
fractional differential systems}

The next theorem contains the main result of this paper regarding the
existence and uniqueness of solution.

\begin{theorem}
\label{main_result1} Suppose that the function $f:\left[ 0,\infty \right)
\times \mathbb{R}^{s}\rightarrow \mathbb{R}^{s}$ is continuous and there
exists a bounded (or continuous) nonnegative function $L:\left[ 0,\infty
\right) \rightarrow \left[ 0,\infty \right) $ such that 
\begin{equation*}
\left\Vert f\left( t,x\right) -f\left( t,y\right) \right\Vert \leq L\left(
t\right) \left\Vert x-y\right\Vert
\end{equation*}%
for all $t\in \left[ 0,\infty \right) $ and $x,y\in \mathbb{R}^{s},$ then
the equation 
\begin{equation}
D_{0_{+}}^{\alpha }x(t)=f(t,x)  \label{3}
\end{equation}%
with the initial condition 
\begin{equation}
\underset{t\rightarrow 0^{+}}{\lim }t^{1-\alpha }x\left( t\right) =x_{0},
\label{4}
\end{equation}%
has a unique solution in the space $C_{1-\alpha }([0,\infty );\mathbb{R}%
^{s}) $ for all $x_{0}\in \mathbb{R}^{s}$.
\end{theorem}

\begin{proof}
To complete the proof of this theorem we only need proving that for any $T>0$
the following integral equation 
\begin{equation}  \label{integral_eq1}
x(t)=\frac{x_{0}}{t^{1-\alpha }}+\frac{1}{\Gamma \left( \alpha \right) }%
\int_{0}^{t}\left( t-\tau\right) ^{\alpha -1}f\left( \tau,x(\tau)\right)
d\tau
\end{equation}
has a unique solution on the interval $[0,T]$ for every $x_0\in\mathbb{R}^s$.%
\newline

For any $x_{0}\in \mathbb{R}^{s}$ let us define an operator $\mathcal{T}%
_{x_{0}}: $ $C_{1-\alpha }\left( \left[ 0,T\right] ,\mathbb{R}^{s}\right)
\rightarrow C_{1-\alpha }\left( \left[ 0,T\right] ,\mathbb{R}^{s}\right) $
by the following formula%
\begin{equation*}
\left( \mathcal{T}_{x_{0}}\xi \right) \left( t\right) =\frac{x_{0}}{%
t^{1-\alpha }}+\frac{1}{\Gamma \left( \alpha \right) }\int_{0}^{t}\left(
t-\tau\right) ^{\alpha -1}f\left( \tau,\xi (\tau)\right) d\tau,\quad \forall
t>0.
\end{equation*}%
This operator is well-defined. Indeed, for any $\xi \in C_{1-\alpha }\left( %
\left[ 0,T\right] ,\mathbb{R}^{s}\right) $, we see 
\begin{align*}
t^{1-\alpha }\int_{0}^{t}(t-\tau )^{\alpha -1}\Vert f(\tau ,\xi (\tau
))\Vert d\tau & \leq Lt^{1-\alpha }\int_{0}^{t}(t-\tau )^{\alpha -1}\Vert
\xi (\tau )\Vert d\tau +t^{1-\alpha }\int_{0}^{t}(t-\tau )^{\alpha -1}\Vert
f(\tau ,0)\Vert d\tau \\
& \leq Lt^{1-\alpha }\int_{0}^{t}(t-\tau )^{\alpha -1}\tau ^{\alpha -1}d\tau
\times \Vert \xi (\tau )\Vert _{1-\alpha ,T} \\
&\hspace{1cm}+ t^{1-\alpha }\sup_{t\in \lbrack 0,T]}\Vert f(t,0)\Vert
\int_{0}^{t}(t-\tau )^{\alpha -1}d\tau \\
& = Lt^{\alpha }B (\alpha ,\alpha )\times \Vert \xi \Vert _{1-\alpha }+\frac{%
\sup_{t\in \lbrack 0,T]}\Vert f(t,0)\Vert }{\alpha }t,
\end{align*}%
where $L:=\underset{t\in \left[ 0,T\right] }{\sup }\Vert L(t)\Vert $, and $%
\Vert \xi \Vert _{1-\alpha ,T}:=\sup_{t\in \lbrack 0,T]}t^{1-\alpha }\Vert
\xi (t)\Vert.$ This shows that%
\begin{equation*}
\lim_{t\rightarrow 0^{+}}t^{1-\alpha }\left( \mathcal{T}_{x_{0}}\xi \right)
\left( t\right) =x_{0}.
\end{equation*}%
Now we consider $t_{0}\in (0,T)$ arbitrarily. For $h>0$ small enough, we
have 
\begin{align*}
& \hspace{-1cm}\left\Vert \int_{0}^{t_{0}+h}(t_{0}+h-\tau )^{\alpha
-1}f(\tau ,\xi (\tau ))d\tau -\int_{0}^{t_{0}}(t_{0}-\tau )^{\alpha -}f(\tau
,\xi (\tau ))d\tau \right\Vert \\
& \leq \int_{t_{0}}^{t_{0}+h}(t_{0}+h-\tau )^{\alpha -1}\Vert f(\tau ,\xi
(\tau ))\Vert d\tau +\int_{0}^{t_{0}}\left( (t_{0}-\tau )^{\alpha
-1}-(t_{0}+h-\tau )^{\alpha -1}\right) \Vert f(\tau ,\xi (\tau ))\Vert d\tau
\\
& =I_{1}(h)+I_{2}(h),
\end{align*}%
where%
\begin{equation*}
I_{1}(h):=\int_{t_{0}}^{t_{0}+h}(t_{0}+h-\tau )^{\alpha -1}\Vert f(\tau ,\xi
(\tau ))\Vert d\tau ,
\end{equation*}%
\begin{equation*}
I_{2}(h):=\int_{0}^{t_{0}}\left( (t_{0}-\tau )^{\alpha -1}-(t_{0}+h-\tau
)^{\alpha -1}\right) \Vert f(\tau ,\xi (\tau ))\Vert d\tau .
\end{equation*}%
By direct computation, 
\begin{align}
I_{1}(h)& \leq \sup_{t\in \lbrack t_{0},t_{0}+h]}\Vert f(t,\xi (t))\Vert
\int_{t_{0}}^{t_{0}+h}(t_{0}+h-\tau )^{\alpha -1}\;d\tau  \notag \\
&= \sup_{t\in \lbrack t_{0},t_{0}+h]}\Vert f(t,\xi (t))\Vert \times \frac{%
h^{\alpha }}{\alpha }.  \label{continuity1}
\end{align}%
Furthermore, 
\begin{align*}
I_{2}(h)& \leq L\Vert \xi \Vert _{1-\alpha ,T}\int_{0}^{t_{0}}\left(
(t_{0}-\tau )^{\alpha -1}-(t_{0}+h-\tau )^{\alpha -1}\right) \tau ^{\alpha
-1}d\tau \\
&\hspace{2cm} +\sup_{t\in \lbrack 0,t_{0}]}\Vert f(t,0)\Vert \int_{0}^{t_{0}}\left(
(t_{0}-\tau )^{\alpha -1}-(t_{0}+h-\tau )^{\alpha -1}\right) d\tau \\
& \leq I_{2,1}(h)+I_{2,2}(h).
\end{align*}%
Note that 
\begin{align*}
I_{2,1}(h)& =L\Vert \xi \Vert _{1-\alpha ,T}\int_{0}^{t_{0}}\left(
(t_{0}-\tau )^{\alpha -1}-(t_{0}+h-\tau )^{\alpha -1}\right) \tau ^{\alpha
-1}d\tau \\
& = L\Vert \xi \Vert _{1-\alpha ,T}\int_{0}^{t_{0}/2}\left( (t_{0}-\tau
)^{\alpha -1}-(t_{0}+h-\tau )^{\alpha -1}\right) \tau ^{\alpha -1}d\tau \\
& \hspace*{1cm}+L\Vert \xi \Vert _{1-\alpha ,T}\int_{t_{0}/2}^{t_{0}} \left(
(t_{0}-\tau )^{\alpha -1}-(t_{0}+h-\tau )^{\alpha -1}\right) \tau ^{\alpha
-1}d\tau \\
& \leq \frac{L\Vert \xi \Vert _{1-\alpha ,T}}{\alpha }\left( (t_{0}-\tau
^\ast )^{\alpha -1}-(t_{0}+h-\tau^ \ast )^{\alpha -1}\right) \left( \frac{%
t_{0}}{2}\right) ^{\alpha } \\
& \hspace*{1cm}+\left( \frac{t_{0}}{2}\right) ^{\alpha }\frac{L\Vert \xi
\Vert _{1-\alpha ,T}}{\alpha }\int_{t_{0}/2}^{t_{0}}\left( (t_{0}-\tau
)^{\alpha -1}-(t_{0}+h-\tau )^{\alpha -1}\right) d\tau \\
& \leq \frac{L\Vert \xi \Vert _{1-\alpha ,T}}{\alpha }\left( (t_{0}-\tau^
\ast )^{\alpha -1}-(t_{0}+h-\tau^ \ast )^{\alpha -1}\right) \left( \frac{%
t_{0}}{2}\right) ^{\alpha } \\
& \hspace*{1cm}+\left( \frac{t_{0}}{2}\right) ^{\alpha }\frac{L\Vert \xi
\Vert _{1-\alpha ,T}}{\alpha }h^{\alpha },
\end{align*}%
for some $\tau^\ast \in (0,t_{0}/2)$. Furthermore, 
\begin{align*}
I_{2,2}(h)& =\sup_{t\in \lbrack 0,t_{0}]}\Vert f(t,0)\Vert
\int_{0}^{t_{0}}\left( (t_{0}-\tau )^{\alpha -1}-(t_{0}+h-\tau )^{\alpha
-1}\right) d\tau \\
& \leq \sup_{t\in \lbrack 0,t_{0}]}\Vert f(t,0)\Vert \times \frac{h^{\alpha }%
}{\alpha }.
\end{align*}%
Thus the function $\mathcal{T}_{x_{0}}\xi (t)$ is continuous on the
right-hand side at $t_{0}$. The left-hand side continuity at $t_{0}$ is
proved similarly, so for any $\xi \in C_{1-\alpha }^{0}([0,\infty ),\mathbb{R%
}^{s})$, the function $\mathcal{T}_{x_{0}}\xi (t)$ is continuous on $(0,T]$.

Let $\gamma >0$ be arbitrary but fixed, we have the estimates%
\begin{align}
\frac{\left\Vert t^{1-\alpha }\left( \mathcal{T}_{x_{0}}\xi -\mathcal{T}%
_{x_{0}}\tilde{\xi }\right) \left( t\right) \right\Vert }{e^{\gamma t}}&\leq 
\frac{t^{1-\alpha }}{\Gamma(\alpha)e^{\gamma t}}\int_{0}^{t}\left( t-\tau\right) ^{\alpha
-1}\frac{L(\tau)\| \xi \left( \tau\right) -\tilde{\xi }\left( \tau\right) \|
\tau^{1-\alpha }\tau^{\alpha -1}e^{\gamma \tau}}{e^{\gamma \tau}}d\tau 
\notag \\
&\leq \frac{Lt^{1-\alpha }}{\Gamma(\alpha)e^{\gamma t}}\| \xi -\tilde{\xi }\| _{w,T
}\int_{0}^{t}\left( t-\tau\right) ^{\alpha -1}\tau^{\alpha -1}e^{\gamma
\tau}d\tau,  \label{5}
\end{align}%
where $\xi, \tilde{\xi}\in C_{1-\alpha}([0,T],\mathbb{R}^s)$ and 
\begin{equation*}
\left\Vert \xi \right\Vert _{w,T }:=\underset{t\in \left[ 0,T\right] }{\sup }%
\frac{t^{1-\alpha }\left\Vert \xi (t)\right\Vert }{e^{\gamma t}}.
\end{equation*}%
Moreover,%
\begin{align*}
\frac{1}{e^{\gamma t}}\int_{0}^{t}\left( t-\tau\right) ^{\alpha
-1}\tau^{\alpha -1}e^{\gamma \tau}d\tau &= \int_{0}^{t/2}\left(
t-\tau\right) ^{\alpha -1}\tau^{\alpha -1}e^{-\gamma \left( t-\tau\right)
}d\tau+\int_{t/2}^{t}\left( t-\tau\right) ^{\alpha -1}\tau^{\alpha
-1}e^{-\gamma \left( t-\tau\right) }d\tau \\
&\leq \frac{2^{1-\alpha }}{t^{1-\alpha }}\int_{0}^{t/2}\tau^{\alpha
-1}e^{-\gamma \tau}d\tau+\frac{2^{1-\alpha }}{t^{1-\alpha }}%
\int_{0}^{t/2}\tau^{\alpha -1}e^{-\gamma \tau}d\tau \\
&\leq \frac{2^{2-\alpha }}{t^{1-\alpha }\gamma ^{\alpha }}\Gamma \left(
\alpha \right) .
\end{align*}%
The last estimate together with (\ref{5}) implies that 
\begin{equation*}
\|\mathcal{T}_{x_0}\xi-\mathcal{T}_{x_0}\hat{\xi}\|_{w,T}\leq \frac{%
L2^{2-\alpha }}{\gamma^\alpha}\times\|\xi-\hat{\xi}\|_{w,T}.
\end{equation*}
Choosing $\gamma $ such that 
\begin{equation*}
\frac{L2^{2-\alpha } }{\gamma ^{\alpha }}<1
\end{equation*}%
implies that the operator $\mathcal{T}_{x_{0}}$ is contractive. By Banach
fixed point theorem we see that the operator $\mathcal{T}_{x_0}$ has a
unique fixed point in $C_{1-\alpha}([0,T];\mathbb{R}^s)$ which is also the
unique solution of the integral equation \eqref{integral_eq1}. The proof is
complete.
\end{proof}

\section{Attractivity of solutions to Riemann--Liouville fractional
differential systems}

Consider the equation 
\begin{equation}
D_{0_{+}}^{\alpha }x(t)=Ax(t)+Q\left( t\right) x(t)+g(t),  \label{2}
\end{equation}
where $A\in\mathbb{R}^{s\times s}$, $Q:[0,\infty)\rightarrow \mathbb{R}%
^{s\times s}$ and $g:[0,\infty)\rightarrow \mathbb{R}^s$ are continuous
functions.

From Theorem \ref{main_result1}, we see that for any $x_0\in \mathbb{R}^s$,
the equation \eqref{2} with the initial condition 
\begin{equation*}
\lim_{t\to 0^+}t^{1-\alpha}x(t)=x_0
\end{equation*}
has a unique global solution in $C_{1-\alpha}([0,\infty),\mathbb{R}^s)$. In
this section, we will study the global attractivity of \eqref{2}. We recall
here this important definition, see e.g. \cite[Definition 2.4]{Chen}.

\begin{definition}
The equation \eqref{2} is called globally attractive if for any $x_0\in%
\mathbb{R}^s$ the solution $\varphi(\cdot,x_0)$ of \eqref{2} such that 
\begin{equation*}
\lim_{t\to 0^+}t^{1-\alpha}\varphi(t,x_0)=x_0
\end{equation*}
tends to zero at infinity, i.e. 
\begin{equation*}
\lim_{t\to \infty}\varphi(t,x_0)=0.
\end{equation*}
\end{definition}

\begin{theorem}
\label{main_result2} Consider the equation \eqref{2}. Suppose that $\sigma{%
(A)}\subset \Lambda_\alpha^s$, the matrix valued function $Q:\left[ 0,\infty
\right) \rightarrow \mathbb{R}^{s\times s}$ satisfies%
\begin{equation*}
\sup_{t\geq 0}\;t^{1-\alpha }\int_{0}^{t}\left( t-\tau\right) ^{\alpha -1}
||| E_{\alpha ,\alpha }\left( \left( t-\tau\right) ^{\alpha }A\right)
Q(\tau)||| \;d\tau<1,
\end{equation*}%
and $g:[0,\infty)\rightarrow\mathbb{R}^s$ is continuous such that 
\begin{equation*}
\sup_{t\geq 0}\;t^{1-\alpha }\int_{0}^{t}\left( t-\tau\right) ^{\alpha -1}
\| E_{\alpha ,\alpha }\left( \left( t-\tau\right) ^{\alpha }A\right)
g(\tau)\| \;d\tau<\infty.
\end{equation*}
Then for any $x_0\in\mathbb{R}^s$, we have $\varphi(\cdot,x_0)\in
C^0_{1-\alpha}([0,\infty),\mathbb{R}^s)$. In particular, the equation %
\eqref{2} is globally attractive.
\end{theorem}

\begin{proof}
It is easy to chack that the conditions of this Theorem implies that the
condition of Theorem \ref{main_result1} in particular there exists a unique
global solution of the initial problem for all initial conditions. Using the
variation-of-constants formula \cite[Theorem 4.2]{Idczak}, we see that the
solution $\varphi \left( \cdot ,x_{0}\right) $ satisfies the following
equation%
\begin{align*}
\varphi \left( t,x_{0}\right)& =t^{\alpha -1}E_{\alpha ,\alpha }\left(
At\right) x_{0}+\int_{0}^{t}\left( t-\tau \right) ^{\alpha -1}E_{\alpha
,\alpha }\left( \left( t-\tau \right) ^{\alpha }A\right) Q(\tau )\varphi
\left( \tau ,x_{0}\right) d\tau\\
&\hspace{2cm}+\int_0^t (t-\tau)^{\alpha-1}E_{\alpha,\alpha}((t-\tau)^\alpha A)g(\tau)\;d\tau
\end{align*}%
for all $t>0.$ For $x_{0}\in \mathbb{R}^{s}$ let us define an operator%
\begin{equation*}
\mathcal{T}_{x_{0}}:C_{1-\alpha }^{0}\left( [0,\infty ),\mathbb{R}%
^{s}\right) \rightarrow C_{1-\alpha }^{0}\left( [0,\infty ),\mathbb{R}%
^{s}\right) 
\end{equation*}%
by the following formula%
\begin{align*}
\left( \mathcal{T}_{x_{0}}\xi \right) \left( t\right)& =t^{\alpha
-1}E_{\alpha ,\alpha }\left( At\right) x_{0}+\int_{0}^{t}\left( t-\tau
\right) ^{\alpha -1}E_{\alpha ,\alpha }\left( \left( t-\tau \right) ^{\alpha
}A\right) Q(\tau )\xi \left( \tau \right) d\tau\\
&\hspace{2cm}+\int_0^t (t-\tau)^{\alpha-1}E_{\alpha,\alpha}((t-\tau)^\alpha A)g(\tau)\;d\tau,\quad \forall t>0.
\end{align*}%
It is obvious that $\mathcal{T}_{x_{0}}\xi \in C_{1-\alpha }^{0}\left(
[0,\infty ),\mathbb{R}^{s}\right) $ for each $\xi \in C_{1-\alpha
}^{0}\left( [0,\infty ),\mathbb{R}^{s}\right) .$ We will show that $\mathcal{%
T}_{x_{0}}$ is a contraction mapping, i.e. there exists a constant $q\in
\left( 0,1\right) $ such that 
\begin{equation*}
\Vert \mathcal{T}_{x_{0}}\xi -\mathcal{T}_{x_{0}}\tilde{\xi}\Vert _{1-\alpha
}\leq q\Vert \xi -\tilde{\xi}\Vert _{1-\alpha }
\end{equation*}%
for all $\xi ,$ $\tilde{\xi}\in C_{1-\alpha }^{0}\left( [0,\infty ),\mathbb{R%
}^{s}\right) $. Indeed, 
\begin{align*}
\left\Vert \left( \mathcal{T}_{x_{0}}\xi -\mathcal{T}_{x_{0}}\tilde{\xi}%
\right) \left( t\right) \right\Vert & \leq \int_{0}^{t}\left( t-\tau \right)
^{\alpha -1}|||E_{\alpha ,\alpha }\left( \left( t-\tau \right) ^{\alpha
}A\right) Q(\tau )|||\times \Vert \xi \left( \tau \right) -\tilde{\xi}\left(
\tau \right) \Vert  \\
& \leq \int_{0}^{t}\left( t-\tau \right) ^{\alpha -1}|||E_{\alpha ,\alpha
}\left( \left( t-\tau \right) ^{\alpha }A\right) Q(\tau )|||\;\tau ^{\alpha
-1}\;d\tau \times \Vert \xi -\tilde{\xi}\Vert _{1-\alpha }
\end{align*}%
and therefore 
\begin{align*}
\Vert \mathcal{\ T}_{x_{0}}\xi -\mathcal{T}_{x_{0}}\tilde{\xi}\Vert
_{1-\alpha }& \leq \Vert \xi -\tilde{\xi}\Vert _{1-\alpha }t^{1-\alpha
}\int_{0}^{t}\left( t-\tau \right) ^{\alpha -1}|||E_{\alpha ,\alpha }\left(
\left( t-\tau \right) ^{\alpha }A\right) Q(\tau )|||\;\tau ^{\alpha -1}d\tau 
\\
& \leq q\Vert \xi -\tilde{\xi}\Vert _{1-\alpha },
\end{align*}%
where%
\begin{equation*}
q:=\sup_{t\geq 0}\;t^{1-\alpha }\int_{0}^{t}|||\left( t-\tau \right)
^{\alpha -1}E_{\alpha ,\alpha }\left( \left( t-\tau \right) ^{\alpha
}A\right) Q(\tau )|||\;\tau ^{\alpha -1}d\tau .
\end{equation*}%
By Banach fixed point theorem there exists a unique fixed point $\xi $ of $%
\mathcal{T}_{x_{0}}$. It is easy to check that this fixed point is the
solution $\varphi (\cdot ,x_{0})$ of the equation (\ref{2}) and, since $\xi
\in C_{1-\alpha }^{0}\left( [0,\infty ),\mathbb{R}^{s}\right) ,$ in
particular 
\begin{equation*}
\underset{t\rightarrow \infty }{\lim }\varphi \left( t,x_{0}\right) =0.
\end{equation*}%
The proof is complete.
\end{proof}

Using Theorem \ref{main_result2}, we obtain immediately the following
corollary.

\begin{corollary}
Consider the equation \eqref{2}. Suppose that $\sigma {(A)}\subset \Lambda
_{\alpha }^{s}$, the function $Q:\left[ 0,\infty \right) \rightarrow \mathbb{%
R}^{s\times s}$ satisfies%
\begin{equation*}
\underset{t\geq 0}{\sup }\;||| Q(t)||| <\frac{1}{\underset{t\geq 0}{\sup }%
\;t^{1-\alpha }\int_{0}^{t}\left( t-\tau\right) ^{\alpha -1}||| E_{\alpha
,\alpha }\left( \left( t-\tau\right) ^{\alpha }A\right) |||\; d\tau},
\end{equation*}%
and $g:[0,\infty )\rightarrow \mathbb{R}^{s}$ is continuous such that 
\begin{equation*}
\sup_{t\geq 0}\;t^{1-\alpha }\int_{0}^{t}\left( t-\tau\right) ^{\alpha -1}||
E_{\alpha ,\alpha }\left( \left( t-\tau\right) ^{\alpha }A\right) g(\tau)||
\;d\tau<\infty .
\end{equation*}%
Then, 
\begin{equation*}
\varphi \left( \cdot ,x_{0}\right) \in C_{1-\alpha }^{0}([0,\infty ),\mathbb{%
R}^{s}),\quad \forall x_{0}\in \mathbb{R}^{s}.
\end{equation*}
\end{corollary}

\begin{theorem}
\label{main_result3} Assume that $\sigma (A)\subset \Lambda _{\alpha }^{s}$, 
$Q:[0,\infty )\rightarrow \mathbb{R}^{s\times s}$ is continuous and
satisfies 
\begin{equation}
\lim_{t\rightarrow \infty }|||Q(t)|||=0,  \label{decay_cond}
\end{equation}%
and $g:[0,\infty )\rightarrow \mathbb{R}^{s}$ is continuous such that 
\begin{equation*}
\sup_{t\geq 0}\;t^{1-\alpha }\int_{0}^{t}\left( t-\tau \right) ^{\alpha
-1}\left\Vert E_{\alpha ,\alpha }\left( \left( t-\tau \right) ^{\alpha
}A\right) g(\tau )\right\Vert d\tau <\infty .
\end{equation*}%
Then for any $x_{0}\in \mathbb{R}^{s}$, the solution $\varphi (\cdot ,x_{0})$
of \eqref{2} satisfies 
\begin{equation*}
\varphi (\cdot ,x_{0})\in C_{1-\alpha }^{0}([0,\infty ),\mathbb{R}^{s}).
\end{equation*}%
In particular, the equation \eqref{2} is globally attractive.
\end{theorem}

\begin{proof}
Let us choose $K>0$ large enough such that 
\begin{equation*}
\sup_{t\geq 0}|||E_{\alpha,\alpha}(t^\alpha A)|||\times \sup_{t\geq
0}|||Q(t)|||\leq K
\end{equation*}
and 
\begin{equation*}
\frac{\sup_{t\geq 0}t^{1-\alpha}\int_0^t
(t-\tau)^{\alpha-1}|||E_{\alpha,\alpha}((t-\tau)^\alpha)A)|||\;\tau^{%
\alpha-1}\;d\tau}{K}\leq \frac{1}{4}.
\end{equation*}
Let $\gamma>0$ be arbitrary but fixed. On the space $C_{1-\alpha}^0([0,%
\infty),\mathbb{R}^s)$, we defined a functional $\|\cdot\|_w$ as below 
\begin{equation*}
\|\xi\|_w:=\sup_{t\geq 0}\frac{t^{1-\alpha}\|\xi(t)\|}{\exp{(\gamma t)}}%
,\quad \forall \xi\in C_{1-\alpha}^0([0,\infty),\mathbb{R}^s).
\end{equation*}
It is obvious to see that $(C_{1-\alpha}^0([0,\infty),\mathbb{R}%
^s),\|\cdot\|_w)$ is also a Banach space.

Now for each $x_0\in \mathbb{R}^s$, on the space $C_{1-\alpha}^0([0,\infty),%
\mathbb{R}^s)$, we construct an operator $\mathcal{T}_{x_0}$ by 
\begin{align*}
\mathcal{T}_{x_0}\xi(t)&:=t^{\alpha -1}E_{\alpha ,\alpha }( t^\alpha A)
x_{0}+\int_{0}^{t}\left( t-\tau\right) ^{\alpha -1}E_{\alpha ,\alpha }\left(
\left( t-\tau\right) ^{\alpha }A\right) Q(\tau)\xi\left( \tau\right) d\tau\\
&\hspace{2cm}+\int_0^t (t-\tau)^{\alpha-1}E_{\alpha,\alpha}((t-\tau)^\alpha A)g(\tau)\;d\tau
\end{align*}%
for all $t>0.$ This operator is well-defined.

Because \eqref{decay_cond}, there exists $T>0$ (large enough) such that 
\begin{equation*}
||| Q(t)||| \leq \frac{1}{K},\quad \forall t\geq T.
\end{equation*}%
Consider the case $t\in \lbrack 0,T]$, we obtain the estimates 
\begin{align*}
\frac{t^{1-\alpha }\Vert \mathcal{T}_{x_{0}}\xi (t)-\mathcal{T}_{x_{0}}%
\tilde{\xi}(t)\Vert }{\exp {(\gamma t)}}& \leq \frac{t^{1-\alpha }}{\exp {%
(\gamma t)}}\int_{0}^{t}(t-\tau )^{\alpha -1}||| E_{\alpha ,\alpha }((t-\tau
)^{\alpha }A)||| \times ||| Q(\tau )||| \times \Vert \xi (\tau )-\tilde{\xi}%
(\tau )\Vert d\tau \\
& \leq K\times t^{1-\alpha }\int_{0}^{t}(t-\tau )^{\alpha -1}\tau ^{\alpha
-1}\exp {(-\gamma (t-\tau ))}d\tau \times \Vert \xi -\tilde{\xi}\Vert _{w} \\
& \leq \frac{K2^{2-\alpha }\Gamma (\alpha )}{\gamma ^{\alpha }}\times \Vert
\xi -\tilde{\xi}\Vert _{w}
\end{align*}%
for any $\xi ,\tilde{\xi}\in C_{1-\alpha }^{0}([0,\infty ),\mathbb{R}^{s})$.
On the other hand, for $t>T$, we have 
\begin{align*}
\frac{t^{1-\alpha }\Vert \mathcal{T}_{x_{0}}\xi (t)-\mathcal{T}_{x_{0}}%
\tilde{\xi}(t)\Vert }{\exp {(\gamma t)}}& \leq \frac{t^{1-\alpha }}{\exp {%
(\gamma t)}}\int_{0}^{t}(t-\tau )^{\alpha -1}||| E_{\alpha ,\alpha }((t-\tau
)^{\alpha }A)||| \times ||| Q(\tau )||| \times \Vert \xi (\tau )-\tilde{\xi}%
(\tau )\Vert \;d\tau \\
& \leq I_{1}(t)+I_{2}(t),
\end{align*}%
where 
\begin{equation*}
I_{1}(t):=\frac{t^{1-\alpha }}{\exp {(\gamma t)}}\int_{0}^{T}(t-\tau
)^{\alpha -1}||| E_{\alpha ,\alpha }((t-\tau )^{\alpha }A)||| \times |||
Q(\tau )||| \times \Vert \xi (\tau )-\tilde{\xi}(\tau )\Vert \;d\tau
\end{equation*}%
and 
\begin{equation*}
I_{2}(t):=\frac{t^{1-\alpha }}{\exp {(\gamma t)}}\int_{T}^{t}(t-\tau
)^{\alpha -1}||| E_{\alpha ,\alpha }((t-\tau )^{\alpha }A)||| \times |||
Q(\tau )||| \times \Vert \xi (\tau )-\tilde{\xi}(\tau )\Vert \;d\tau .
\end{equation*}%
By using the same arguments as in the proof of Theorem \ref{main_result1},
we obtain 
\begin{equation*}
I_{1}(t)\leq \frac{K2^{2-\alpha }\Gamma (\alpha )}{\gamma ^{\alpha }}\times
\Vert \xi -\tilde{\xi}\Vert _{w}.
\end{equation*}%
For the quantity $I_{2}(t)$, we estimate 
\begin{align*}
I_{2}(t)& \leq t^{1-\alpha }\int_{T}^{t}(t-\tau )^{\alpha -1}||| E_{\alpha
,\alpha }((t-\tau )^{\alpha }A)||| \tau ^{\alpha -1}\exp {(-\gamma (t-\tau ))%
}\;d\tau \times \frac{1}{K}\times \Vert \xi -\tilde{\xi}\Vert _{w} \\
& \leq \frac{\sup_{t\geq 0}t^{1-\alpha }\int_{0}^{t}(t-\tau )^{\alpha -1}|||
E_{\alpha ,\alpha }((t-\tau )^{\alpha }A)||| \tau ^{\alpha -1}\;d\tau }{K}%
\times \Vert \xi -\tilde{\xi}\Vert _{w}.
\end{align*}%
Thus, 
\begin{equation*}
\Vert \mathcal{T}_{x_{0}}\xi -\mathcal{T}_{x_{0}}\tilde{\xi}\Vert _{w}\leq
\left( \frac{K2^{2-\alpha }\Gamma (\alpha )}{\gamma ^{\alpha }}+\frac{%
\sup_{t\geq 0}t^{1-\alpha }\int_{0}^{t}(t-\tau )^{\alpha -1}||| E_{\alpha
,\alpha }((t-\tau )^{\alpha }A)||| \tau ^{\alpha -1}\;d\tau }{K}\right)
\times \Vert \xi -\tilde{\xi}\Vert _{w}
\end{equation*}%
for all $\xi ,\tilde{\xi}\in C_{1-\alpha }^{0}([0,\infty ),\mathbb{R}^{s})$.
This implies that $\mathcal{T}_{x_{0}}$ is contractive in $(C_{1-\alpha
}^{0}([0,\infty ),\mathbb{R}^{s}),\Vert \cdot \Vert _{w})$ if we choose $%
\gamma >0$ satisfies 
\begin{equation*}
\frac{K2^{2-\alpha }\Gamma (\alpha )}{\gamma ^{\alpha }}\leq \frac{1}{4}.
\end{equation*}%
By Banach fixed point theorem, the operator $\mathcal{T}_{x_{0}}$ has a
unique fixed point in $C_{1-\alpha }^{0}([0,\infty ),\mathbb{R}^{s})$ which
is also the unique solution $\varphi (\cdot ,x_{0})$ of \eqref{2} on $%
[0,\infty )$. Because $\varphi (\cdot ,x_{0})\in C_{1-\alpha }^{0}([0,\infty
),\mathbb{R}^{s})$, the proof is complete.
\end{proof}

\section{Example}

In this section, we give some examples to illustrate for the theoretical
results above.

\begin{example}
Consider the equation 
\begin{equation}
D_{0+}^{1/2}x(t)=-x(t)+Q(t)x(t)+g(t),  \label{Exam1}
\end{equation}%
where $Q:[0,\infty )\rightarrow \mathbb{R}$ is continuous and satisfies 
\begin{equation*}
|Q(t)|\leq \frac{1}{\sup_{t\geq 0}t^{1/2}E_{1/2}(-t^{1/2})},\quad \forall
t>0,
\end{equation*}%
and $g(t)=\frac{1}{1+t^{1/2}}$ for all $t\geq 0$. Then for any $x_{0}\in 
\mathbb{R}$, the solution $\varphi (\cdot ,x_{0})$ of \eqref{Exam1} with the
initial condition $\lim_{t\rightarrow 0^{+}}t^{1/2}x(t)=x_{0}$ converges to
zero as $t$ tends to infinity. Indeed, from the assumption on the
function $Q(\cdot )$, we have 
\begin{align*}
& \left\vert t^{1/2}\int_{0}^{t}(t-\tau )^{-1/2}E_{1/2,1/2}\left( -(t-\tau
)^{1/2}Q(\tau )\right) \;d\tau \right\vert  \\
& \hspace{2cm}\leq \frac{t^{1/2}}{\sup_{t\geq 0}t^{1/2}E_{1/2}(-t^{1/2})}%
\int_{0}^{t}(t-\tau )^{-1/2}E_{1/2,1/2}(-(t-\tau )^{1/2})\tau ^{-1/2}\;d\tau 
\\
& \hspace{2cm}\leq \frac{t^{1/2}}{\sup_{t\geq 0}t^{1/2}E_{1/2}(-t^{1/2})}%
\times E_{1/2}(-t^{1/2}),
\end{align*}%
see \cite[Formula (1.100), p.~25]{Podlubny}. Hence, 
\begin{equation*}
\left\vert t^{1/2}\int_{0}^{t}(t-\tau )^{-1/2}E_{1/2,1/2}\left( -(t-\tau
)^{1/2}Q(\tau )\right) \;d\tau \right\vert <1.
\end{equation*}%
On the other hand, 
\begin{align*}
& \left\vert t^{1/2}\int_{0}^{t}(t-\tau )^{-1/2}E_{1/2,1/2}\left( -(t-\tau
)^{1/2}g(\tau )\right) \;d\tau \right\vert  \\
& \hspace{2cm}\leqslant \left\vert t^{1/2}\int_{0}^{t}(t-\tau
)^{-1/2}E_{1/2,1/2}\left( -(t-\tau )^{1/2}\tau ^{-1/2}\right) \;d\tau
\right\vert  \\
& \hspace{2cm}=t^{1/2}E_{1/2}(-t^{1/2})<\infty 
\end{align*}%
for all $t\geq 0$. Following Theorem \ref{main_result2}, we see that for any 
$x_{0}\in \mathbb{R}$, the solution $\varphi (\cdot ,x_{0})$ of \eqref{Exam1}
with the initial condition $\lim_{t\rightarrow 0^{+}}t^{1/2}x(t)=x_{0}$
converges to zero as $t$ tends to infinity.
\end{example}

\begin{example}
Consider the equation 
\begin{equation}  \label{Exam2}
D^{1/2}_{0+}x(t)=-x(t)+Q(t)x(t)+g(t),
\end{equation}
where $$Q(t):=\begin{cases}
&1000,\quad t\in [0,1000],\\
&\frac{1000^2}{t},\quad t\in [1000,\infty),
\end{cases}$$ and $g(t)=\frac{1}{1+t^{1/2}}$
for all $t\geq 0$. In this case, it is easy to see that the conditions in
Theorem \ref{main_result3} are satisfied. Thus, for any $x_0\in\mathbb{R}$,
the solution $\varphi(\cdot,x_0)$ of \eqref{Exam2} with the initial
condition $\lim_{t\to 0^+}t^{1/2}x(t)=x_0$ decays at infinity.
\end{example}

\section{Conclusions}

In this paper, we have studied the existence and attractivity of solution of
nonlinear differential system with Riemann-Liouville fractional derivative.
Using Banach fixed point theorem and properties of Mittag-Leffler functions
we obtained some sufficient conditions for these properties.

\section{Acknowledgment}

The research of H.T. Tuan is supported by the Vietnam National Foundation
for Science and Technology Development (NAFOSTED). The research of Adam
Czornik and Micha\l {} Niezabitowski was done as parts of the projects
funded by the National Science Centre in Poland granted according to
decisions DEC-2012/07/N/ST7/03236 and DEC-2015/19/D/ST7/03679, respectively.
Moreover, the calculations were performed with the use of IT infrastructure
of GeCONiI Upper Silesian Centre for Computational Science and Engineering
(NCBiR grant no POIG.02.03.01-24-099/13). The research of J.J. Nieto is
partially supported by Agencia Estatal de Innovaci\'on (AEI) of Spain,
project MTM2016-75140-P, cofinanced by FEDER and Xunta de Galicia, grants
GRC 2015\^{a}\euro ``004 and R 2016/022.

\end{document}